\documentclass[12pt]{amsart}
\usepackage{amssymb, amsmath}
\title{On conjugate separability of nilpotent subgroups}
\author{M. Shahryari}
\address{M. Shahryari: Department of  Mathematics,  College of Science, Sultan Qaboos University, Muscat, Oman}
\email{m.ghalehlar@squ.edu.om}

\newtheorem{example}{Example}
\newtheorem{corollary}{Corollary}
\newtheorem{proposition}{Proposition}

\newtheorem {theorem}{Theorem}
\newtheorem{lemma}{Lemma}
\newtheorem{remark}{Remark}

\numberwithin{equation}{section}

\newcommand{\NTk}{\mathrm{NT}_k}
\newcommand{\CSNk}{\mathrm{CSN}_k}
\newcommand{\nilk}{\mathfrak{N}_k}

\begin{document}

\maketitle
\begin{abstract}
We study  groups, all maximal nilpotent subgroups of class at most $k$ in which are malnormal. We show that such groups share many similar properties with the  ordinary CSA groups. Similarly, we introduce the class of {\em nilpotency transitive} groups $\NTk$,  and we show that in presence of a special residuality condition, these two concepts are  equivalent. As a result, we see that the theory of CSA and CT groups is a small part of a more general idea.
\end{abstract}
\vspace{1cm}


\vspace{1cm}

A group $G$ is called CSA (conjugately  separated abelian) if every maximal abelian subgroup  of $G$ is {\em malnormal}. This means that if $H$ is a maximal abelian subgroup of $G$ and $x\in G\setminus H$ then $H\cap H^x=1$. The class of CSA groups is quite wide and has very serious roles in the study of residually free groups, universal theory of non-abelian free groups, limit groups, exponential groups and equational domains in algebraic geometry over groups (see \cite{BMR}, \cite{Champ}, \cite{MR1} , and \cite{MR2}). Another class of groups which has been studied extensively is the class of CT (commutative transitive) groups. A group is CT if commutativity is a transitive relation on the set of its non-identity elements. Despite this simple definition, the class of CT groups has also a crucial role in the study of residually free groups and so it has a close connection with CSA groups. Every CSA group is CT but the converse is not true. In the presence of residual freeness, both properties are equivalent, a theorem which is proved by B. Baumslag (see \cite{BB}).

During the past few decades, there have been many attempts to study these classes and their generalizations. A generalization of CT groups is introduced in \cite{Ciob} to extend the above mentioned theorem of B. Baumslag. Many interesting relations between CSA and CT groups are presented in \cite{Fine} as well as  an excellent account of the previous works.

It seems that the idea of CT and CSA groups is a small part of a very general concept. Suppose $\mathfrak{X}$ is a variety of  groups (it can even be a universal class or even an inductive class of groups closed under subgroup). A group $G$ can be called $\mathfrak{X}T$ then,  if and only if for any two $\mathfrak{X}$-subgroups $K_1, K_2\leq G$ the assumption $K_1\cap K_2\neq 1$ implies that $\langle K_1, K_2\rangle$ is also an $\mathfrak{X}$-group. Similarly, we  call a group $G$ a $CS\mathfrak{X}$ group if  all of its maximal $\mathfrak{X}$-subgroups  are malnormal. 

Although it seems that most parts of the work in this article can be developed  for many general classes $\mathfrak{X}$, we focus only on the variety of nilpotent groups of class at most $k$. Let's denote this variety by $\nilk$. Hence, we call a group $\NTk$ (nilpotency transitive of class $k$) if for any two $\nilk$-subgroups $K_1$ and $K_2$, the assumption $K_1\cap K_2\neq 1$ implies that $ \langle K_1, K_2\rangle$ is nilpotent of class at most $k$. Also  a group $G$ is $\CSNk$ (conjugately separated nilpotent of class $k$) if and only if every maximal $\nilk$-subgroup of $G$ is malnormal. The case $k=1$  obviously coincides with the ordinary CT and CSA groups. However, there is no implications of the form $\NTk \rightarrow CT$ or $\CSNk \rightarrow CSA$ (the second implication is not true as not every maximal abelian subgroup is  necessarily a maximal $\nilk$-subgroup).  Despite this difference, we will  show that these classes share many similar properties with the classical cases of CT and  CSA groups. After a detailed study of basic properties of $\NTk$ and $\CSNk$ groups, as the main result,  we will show that in presence of a special residuality condition, these two classes are the same. In order to do this, we will introduce a natural generalization of the concept of free group. Suppose $A$ is a finitely generated free element of the variety $\nilk$. Then every free product of a set of copies of $A$ will be called an $A$-free group. Obviously, if $A=\mathbb{Z}$ then we get the ordinary free groups.   We will show that in the class of residually $A$-free groups, the notions of $\NTk$ and $\CSNk$ groups are equivalent.

A few words are needed here to clarify our notation: all simple commutators $[ x_1, x_2, \ldots, x_{k+1}]$ are left aligned. The subgroup generated by a subset $X\subseteq G$ will be denoted by $\langle X\rangle$ and the normal closure of this subgroup is $\langle X^G\rangle$. A conjugate $a^x$ (or $H^x$) is $x^{-1}ax$ (similarly, $x^{-1}Hx$). For a class $\{ G_i\}_{i\in I}$ of groups, the corresponding free product will be denoted by $\ast_{i\in I}G_i$. For a group $G$ the notation $\mathrm{Th}_{\forall}(G)$ will be used for the set of all universal first order sentences true in $G$.

\section{Basic properties}
We begin this section with the following natural relation.

\begin{proposition}\label{first}
Every $\CSNk$ group is $\NTk$.
\end{proposition}

\begin{proof}
Suppose $G$ is a $\CSNk$ group and $K_1, K_2\leq G$ are $\nilk$-subgroups with $K_1\cap K_2\neq 1$. By Zorn's lemma there is a maximal $\nilk$-subgroup $H$ containing $K_1$. We prove that $K_2\subseteq H$. Suppose there exists $z\in K_2\setminus H$. We have $H\cap H^z=1$ as $H$ is malnormal. Let $y\in K_1\cap K_2$ be a non-identity element. As $\langle y, z\rangle \subseteq K_2$ is $\nilk$-group, we have
$[z,_k y]=1$ and hence
$$
[z,_{k-1} y]^{-1}y^{-1}[z,_{k-1} y]y=1.
$$
This implies that $[z,_{k-1} y]^{-1}y^{-1}[z,_{k-1} y]=y^{-1}\in H$. Consequently
$$
H\cap H^{[z,_{k-1} y]}\neq 1
$$
and therefore $[z,_{k-1} y]\in H$. Similar argument shows that $[z,_{k-2} y]\in H$ and if we continue, we see that $[z, y]\in H$. In other words $y\in H\cap H^z$ and hence $z\in H$. This shows that $K_2\subseteq H$ and hence $\langle K_1, K_2\rangle \subseteq H$ is $\nilk$-subgroup.
\end{proof}

The next result shows that non-$\nilk$-groups with the property $\NTk$ are indecomposable.

\begin{proposition}
Suppose $G$ is an $\NTk$ group. Then\\

1- For every pair of distinct maximal $\nilk$-subgroups $H_1$ and $H_2$ we have $H_1\cap H_2=1$.

2- The group $G$ is directly indecomposable except in the case when it is $\nilk$.
\end{proposition}

\begin{proof}
To prove part 1, suppose $H_1\cap H_2\neq 1$. Then $\langle H_1, H_2\rangle$ is $\nilk$-subgroup and by maximality $H_1=\langle H_1, H_2\rangle =H_2$. For the part 2, suppose $G=A\times B$ for some non-trivial subgroups $A$ and $B$. Let $x_1, \ldots, x_{k+1}\in A$ and $1\neq y\in B$ be arbitrary elements. For any $i$ we have $[x_i, y]=1$ and hence all subgroups $\langle x_i, y\rangle$ are abelian. This shows that  $\langle x_1, \ldots, x_{k+1}, y\rangle$ is a $\nilk$-subgroup and hence $[x_1, \ldots, x_{k+1}]=1$. Hence, $A$ is a $\nilk$-group. Similarly $B$ is a $\nilk$-group and so is $G$.
\end{proof}

The converse of the part 1 is also true.

\begin{proposition}
Suppose for every pair of distinct maximal $\nilk$-subgroups $H_1$ and $H_2$ we have $H_1\cap H_2=1$. Then $G$ is $\NTk$.
\end{proposition}

\begin{proof}
Suppose $K_1$ and $K_2$ are $\nilk$-subgroups of $G$ and $K_1\cap K_2\neq 1$. Let $H_i$ be a maximal $\nilk$-subgroup containing $K_i$ for $i=1, 2$. Then $K_1\cap K_2\subseteq H_1\cap H_2$ and so $H_1=H_2$. Therefore $\langle K_1, K_2\rangle \subseteq H_1$ and hence it is a $\nilk$-subgroup, showing that $G$ is $\NTk$.
\end{proof}

Centralizers play a crucial role in the study of ordinary CT and CSA groups. One may see the important link between CSA groups and extension by the centralizers in the fundamental works of Myasnikov and Remeslennikov (see \cite{MR1} and \cite{MR2}). In the case of $\NTk$ groups, we need the following definition. For every element $x$ in a group $G$ we define a subset
$$
C^k_G(x)=\{ y\in G:\ \langle x, y\rangle \in \nilk\}.
$$
Note that this set contains the centralizer $C_G(x)$ but of course it is not a subgroup in general. However, the situation will be changed in the presence of the property $\NTk$.

\begin{proposition}
A group $G$ is $\NTk$ if and only if for any $x\neq 1$, the set $C^k_G(x)$ is a $\nilk$-subgroup. If $G$ is $\NTk$ then for all $x\neq 1$ the subgroup $C^k_G(x)$ is a maximal $\nilk$-subgroup and every maximal $\nilk$-subgroup has this form.
\end{proposition}

\begin{proof}
Let $G$ be $\NTk$ and $x\neq 1$. First we show that $C^k_G(x)$ is a subgroup. Suppose $y_1, y_2\in C^k_G(x)$. Then the subgroups $\langle x, y_1\rangle$ and $\langle x, y_2\rangle$ are $\nilk$ with non-trivial intersection. Hence $\langle x, y_1, y_2\rangle$ is a $\nilk$-subgroup which shows that $\langle x, y_1^{-1}y_2\rangle$ is also $\nilk$. This means that $C^k_G(x)$ is a subgroup. To prove that it is a $\nilk$-subgroup, suppose $x_1, \ldots, x_{k+1}\in C^k_G(x)$. Then using the $\NTk$ property of $G$ the subgroup $\langle x, x_1, \ldots, x_{k+1}\rangle$ must be $\nilk$ and hence $C^k_G(x)$ is a $\nilk$-subgroup.

Conversely, suppose for all $x\neq 1$, the set $C^k_G(x)$ is a $\nilk$-subgroup. We show that $G$ is $\NTk$. Let $K_1$ and $K_2$ be $\nilk$-subgroups of $G$ with non-trivial intersection. Suppose $x\in K_1\cap K_2$ is a non-identity element. Then for all $y\in K_1$ we have $\langle x, y\rangle \subseteq K_1$ and hence it is $\nilk$, showing that $y\in C^k_G(x)$. As a result both $K_1$ and $K_2$ are contained in $C^k_G(x)$ and hence $\langle K_1, K_2\rangle$ is $\nilk$-subgroup. Consequently, $G$ is $\NTk$.

Now, suppose $G$ is $\NTk$ and $x\neq 1$. Let $C^k_G(x)\subseteq K$ where $K$ is a $\nilk$-subgroup. If $y\in K$ then $\langle x, y\rangle \subseteq K$ and therefore $y\in C^k_G(x)$. Hence $C^k_G(x)$ is a maximal $\nilk$-subgroup. To show that every maximal $\nilk$-subgroup has this form, let $H$ be an arbitrary maximal $\nilk$-subgroup. Let $x\in H$ be non-identity. Then for all $y\in H$ the subgroup $\langle x, y\rangle$ is $\nilk$ and hence $y\in C^k_G(x)$ which shows that $H=C^k_G(x)$.
\end{proof}

As a result we have the following.

\begin{proposition}
Let $G$ be a $\CSNk$ group and $H$ be a subgroup. Then $H$ is also $\CSNk$.
\end{proposition}

\begin{proof}
Let $K\leq H$ be a maximal $\nilk$-subgroup. As $H$ is $\NTk$, we have $K=C^k_H(x)$ for some non-identity $x\in H$. Let $y\in H\setminus K$ and $K\cap K^y\neq 1$. Let $M\leq G$ be a maximal $\nilk$-subgroup which contains $K$. Then $M\cap M^y$ contains $K\cap K^y$ and hence $y\in M$ as $G$ is $\CSNk$. On the other hand, $x\in K\subseteq M$, so $\langle x, y\rangle$ is a $\nilk$-subgroup which means that $y\in C^k_H(x)=K$, a contradiction. Therefore $K\cap K^y=1$ and $H$ is $\CSNk$.
\end{proof}

More is true for the classes of $\NTk$ and $\CSNk$ groups: they are both axiomatizable by universal sentences. In order to show this, first we need an easy observation from elementary group theory.
Suppose $X$ is an arbitrary subset of a group $G$. By the notation $[X,_k X]$ we denote the set of all simple commutators of length $k+1$ made by the elements of $X$. A simple argument shows that if $G=\langle X\rangle$ then $\gamma_{k+1}(G)=\langle [X,_k X]^G\rangle$ where $\gamma_{k+1}(G)$ is the $(k+1)$-th term of the lower central series. . This implies that $G$ is $\nilk$ if and only if for every generating subset $X$ we have $[X,_k X]=1$.

Now, for any elements $x$ and $y$ in a group $G$, suppose $Q(x, y)$ is the first order sentence
$$
\bigwedge_{x_1, \ldots, x_{k+1}\in \{ x, y\}}[x_1, \ldots, x_{k+1}]\approx 1.
$$
Note that we prefer to use the notation $\approx$ for equality in the first order language of groups rather the ordinary $=$. By the above observation the sentence $Q(x,y)$ is true in $G$ if and only if $\langle x, y\rangle$ is $\nilk$.

Using this notation  the following sentence
$$
(\mathrm{Subgp})\ \ \ \ \ \ \forall x \forall y_1, y_2:((x\not\approx 1\wedge Q(x, y_1)\wedge Q(x, y_2))\to Q(x, y_1^{-1}y_2))
$$
says that for all non-identity element $x$ the subset $C^k_G(x)$ is a subgroup and similarly the sentence
$$
(\mathrm{Nil})\ \forall x \forall y_1, \ldots, y_{k+1}:((x\not\approx 1)\wedge \bigwedge_{i=1}^{k+1}Q(x, y_i))\to [y_1, \ldots, y_{k+1}]\approx 1)
$$
means that $C^k_G(x)$ is nilpotent of class at most $k$ for each non-identity $x$. Consequently, the property of being $\NTk$ can be translated to the universal first order sentence $\mathrm{Subgp}+\mathrm{Nil}$. In the case of $\CSNk$ groups every  maximal $\nilk$-subgroup has the form of $C^k_G(x)$ for some non-trivial element $x$, so $G$ is $\CSNk$ if and only if it is $\NTk$ and for all $x\neq 1$ the subgroup $C^k_G(x)$ is malnormal. Hence, if we consider the sentence
$$
(\mathrm{Mal})\ \ \ \ \ \ \ \ \ \ \ \forall x, y, z:((x, y\not\approx 1\wedge Q(x, y)\wedge Q(x, y^z))\to Q(x, z))
$$
then the property of being $\CSNk$ can be described by the universal sentence $\mathrm{Subgp}+\mathrm{Nil}+\mathrm{Mal}$. As a result we have

\begin{proposition}
The classes of $\NTk$ and $\CSNk$  are universal. Hence any ultraproduct of $\NTk$ groups is $\NTk$ and any ultraproduct of $\CSNk$ groups is $\CSNk$.
\end{proposition}

Here we have an example of construction of new $\NTk$ groups. The similar construction of $\CSNk$ group will be described later.

\begin{example}
Suppose  $A$ and $B$ are two $\NTk$ groups. Let $G=A\ast B$ be the free product of $A$ and $B$. If $K\leq G$ then by  Kurosh subgroup theorem we have
$$
K=F[X]\ast (\ast_{i\in I}A_i^{u_i})\ast (\ast_{j\in J}B_j^{v_j})
$$
where $X$ is a subset of $G$ and $F[x]$ is the free group  generated by $X$, each $A_i$ is a subgroup of $A$ and each $B_j$ is a subgroup of $B$, and for all $i$ and $j$ we have $u_i, v_j\in G$. Suppose $K$ is an $\nilk$-group. Then we only have the following cases.\\

1- $K$ is an infinite cyclic subgroup.

2- $K=A_i^{u_i}$ for some $i\in I$.

3- $K=B_j^{v_j}$ for some $j\in J$.\\

As a result, if $K_1$ and $K_2$ are $\nilk$-subgroups of $G$ with non-trivial intersection, then in all possible combinations of the above three cases, $\langle K_1, K_2\rangle$ will be a $\nilk$-group, showing that $G$ is $\NTk$. Consequently, for any set $\Lambda$ and any ultrafilter $\mathcal{U}$ over $\Lambda$, the ultrapower
$$
G^{\ast}=\frac{(A\ast B)^{\Lambda}}{\mathcal{U}}
$$
is also $\NTk$.
\end{example}

There is one more elementary property of $\NTk$ groups. The proof is straightforward and we skip it.

\begin{proposition}
Let $G$ be an $\NTk$ group. Then we have the following.\\

1- If $G$ is torsion-free and $x^m=y^n$ for some elements $x, y\in G$ and some integers $m$ and $n$, then $\langle x, y\rangle$ is a $\nilk$-subgroup.

2- If $G$ is torsion-free and $x^n=y^n$ for some elements $x, y\in G$ and some integer $n$, then $x=y$.

3- If $G$ is not a $\nilk$-group then $Z(G)=1$, where $Z(G)$ denotes the center of $G$ and hence $G$ is not nilpotent (of any class).
\end{proposition}

The next result shows that the class $\CSNk$ is also closed under free product if we avoid groups containing involutions. This is exactly the same property as in \cite{MR2} for CSA groups and to prove it we only need to mimic the proof in the CSA case.

\begin{proposition}\label{freeprod}
Suppose $A$ and $B$ are $\CSNk$ groups without elements of order 2. Then the free product $G=A\ast B$ is also $\CSNk$.
\end{proposition}

\begin{proof}
Suppose $H\leq G$ is a maximal  $\nilk$-subgroups. We show that $H$ is malnormal in $G$. Let $x\in G$ be an element with the property $H\cap H^x\neq 1$. Using Kurosh subgroup theorem we have the following cases. \\

1- $H\leq A^y$ for some element $y\in G$. Using a suitable conjugation we may assume that $H\leq A$. We know that if $x$ does not belong to $A$ then $A\cap A^x=1$. Hence $H\cap H^x\neq 1$ implies that $x\in A$. As $A$ was $\CSNk$, we see that $x\in H$. \\

2- $H\leq B^y$ for some element $y\in G$. This case is similar to the previous one.\\

3- $H=\langle z\rangle$ for some element $z$ of infinite order which has cyclically reduced length bigger than one. As $H\cap H^x\neq 1$, we have $x^{-1}z^nx=z^m$ for some integers $m$ and $n$. Comparing the cyclically reduced lengths we see that $m=\pm n$. If $m=n$ then $xz^n=z^nx$ and so $[x, z^n]=1$. But $H$ is a maximal cyclic subgroup, hence $x\in \langle z\rangle=H$.

Let $m=-n$. Then $x^{-1}z^nx=z^{-n}$ so $[x^2, z^n]=1$. As $x^2\neq 1$ we have again $x\in \langle z\rangle=H$.
\end{proof}

Is there any $\NTk$ group which is not $\CSNk$? The next example shows that if $A$ and $B$ have involutions then $A\ast B$ will not be $\CSNk$.

\begin{example}
Suppose $A$ and $B$ are $\nilk$-groups having elements $x\in A$ and $y\in B$ of orders $2$. Then the subgroup $H=\langle xy\rangle$ is a maximal $\nilk$-subgroup of $A\ast B$. We have $x^{-1}(xy)x=yx=(xy)^{-1}$ and hence $H\cap H^x\neq 1$. Therefore $G=A\ast B$ is not $\CSNk$ while it is still $\NTk$.
\end{example}

It is known that every finite CSA group is abelian. The structure of finite CT groups are already determined by many people.  Here we show that finite $\CSNk$ groups are nilpotent of class at most $k$.

\begin{proposition}
Every finite $\CSNk$ group is nilpotent of class at most $k$.
\end{proposition}

\begin{proof}
Suppose $G$ is a finite $\CSNk$ group but not nilpotent. Let $H$ be a maximal $\nilk$-subgroup. As $H$ is malnormal, in the language of finite group theory, it is a Frobenius  group with a complement $H$. So there exists a Frobenius kernel $N$ which is a normal nilpotent subgroup of $G$ and we have the semidirect decomposition $G=H\ltimes N$ (see \cite{Isaacs} for details on Frobenius groups). Suppose $n$ is  the nilpotency class of of $N$. We consider two cases.\\

1- If $n\leq k$ then $N\subseteq M$ for some maximal $\nilk$-subgroup $M$. As $N$ is normal in $G$, for every $x\in G\setminus M$ we have $N^x=N$ and hence $M\cap M^x\neq 1$ which is impossible.\\

2- If $n>k$ then we consider the  $N_0=\gamma_{n-k}(N)$ which is a normal $\nilk$-subgroup of $G$. Again if we suppose that $N_0\subseteq M$ for some maximal $\nilk$-subgroup $M$, then $M\cap M^x$ will be non-trivial for every $x\in G\setminus M$, a contradiction. \\

As a result the group $G$ must be nilpotent. A similar argument as above shows that the class of $G$ is at most $k$.
\end{proof}

In \cite{Ciob} it is shown that a CT group $G$ is not CSA if and only if it has a non-ableian subgroup which contains a non-trivial abelian normal  subgroup. We show that this can be generalized to the case of $\NTk$ groups.

\begin{theorem}
An $\NTk$ group $G$ is not $\CSNk$ if and only if it has a subgroup $G_0$  which is not $\nilk$ and  $G_0$ itself contains a non-trivial $\nilk$-subgroup which is normal in $G_0$.
\end{theorem}

\begin{proof}
First suppose that $G$ has a subgroup $G_0$ which is not $\nilk$ but contains a non-trivial $\nilk$-subgroup $A$ which is normal in $G_0$. Let $M$ be a maximal $\nilk$-subgroup of $G_0$ containing $A$. Then for all $x\in G_0\setminus M$ the intersection $M\cap M^x$ contains $A$ and so it is non-trivial. This shows that $G_0$ (and consequently $G$)  is not  $\CSNk$.

Now suppose that $G$ is $\NTk$ but not $\CSNk$.  This means that $G$ satisfies the first order sentences $\mathrm{Subgp}$ and $\mathrm{Nil}$ but not $\mathrm{Mal}$. Consequently, there are element $x, y, z\in G$ such that $x, y\neq 1$ and $Q(x, y)$ and $Q(x, z^{-1}yz)$ are true but $Q(x, z)$ is false. In other words $\langle x, y\rangle$ and $\langle x, z^{-1}yz\rangle$ are $\nilk$-subgroups but $\langle x, z\rangle$ is not. Suppose $G_0=\langle x, y, z\rangle$. Then $G_0$ is not a $\nilk$-subgroup. Let $A=\langle y^{G_0}\rangle$ be the normal closure of $\langle y\rangle$ in $G_0$. We prove that $A$ is a $\nilk$-subgroup. Note that
$$
A=\langle u^{-1}yu:\ u\in G_0\rangle.
$$
As $G$ is $\NTk$ and $\langle x, y\rangle$ and $\langle x, z^{-1}yz\rangle$ are $\nilk$, we conclude that the subgroup $\langle x, y, z^{-1}yz\rangle$ is also a $\nilk$. Conjugating with $x$, $y$, and $z$, respectively, we see that the following subgroups are also $\nilk$:
$$
\langle x, x^{-1}yx, x^{-1}z^{-1}yzx\rangle
$$
$$
\langle y^{-1}xy, y, y^{-1}z^{-1}yzy\rangle
$$
$$
\langle z^{-1}xz, z^{-1}yz, z^{-2}yz^2\rangle
$$
Applying the assumption of being $\NTk$ again the following subgroups are also $\nilk$:
$$
\langle x, y, x^{-1}yx, z^{-1}yz, x^{-1}z^{-1}yzx\rangle
$$
$$
\langle x, y, y^{-1}xy,  z^{-1}yz, y^{-1}z^{-1}yzy\rangle
$$
$$
\langle x, y, z^{-1}xz, z^{-1}yz, z^{-2}yz^2\rangle
$$
and so is the subgroup
$$
\langle y, x^{-1}yx, z^{-1}yz, x^{-1}z^{-1}yzx, y^{-1}z^{-1}yzy, z^{-2}yz^2\rangle.
$$
Conjugating again by $x$, $y$, and $z$ and using $\NTk$ assumption we see that the subgroup
$$
\langle u^{-1}yu:\ u\in G_0, |u|\leq 2\rangle
$$
is $\nilk$. Note that $|u|$ denotes the length of $u$ with respect to the generating set $S=\{ x, y, z\}$. We can continue this argument to show that the subgroup
$$
\langle u^{-1}yu:\ u\in G_0, |u|\leq n\rangle
$$
is $\nilk$ for every $n$. As a result $A$ itself is a $\nilk$-subgroup.
\end{proof}

\section{$A$-free and residually $A$-free groups}
Fix a finitely generated free element of the variety $\nilk$, say $A$. The free product of any non-empty family of copies of $A$ will be called an $A$-free group. For example, in the most trivial case, the concepts of  $\mathbb{Z}$-free group and  ordinary free group are equivalent. Note that every $A$-free group is $\CSNk$ by the proposition \ref{freeprod} as $A$ is torsion-free. A group $G$ is called residually $A$-free if for every non-identity element $g\in G$ there exists a homomorphism $\alpha$ from $G$ to some $A$-free group such that $\alpha(g)\neq 1$. It is called fully residually $A$-free if and only if for any finite set of non-identity elements $g_1, \ldots, g_n\in G$ there exists a homomorphism $\alpha$
from $G$ to some $A$-free group such that $\alpha(g_i)\neq 1$ for all $1\leq i\leq n$.

\begin{remark}
Thanks to a comment of Lee Mosher, for every $m$, the $A$-free group $\ast^m A$ can be embedded in  $A\ast A$: let $B$ and $C$ be two groups isomorphic to $A$. Then $A\ast A\cong B\ast C$. Suppose $a_1, \ldots, a_{m-1}\in A$ are distinct non-identity elements of $A$. Then the subgroup generated by $B, a_1^{-1}Ca_1, \ldots, a_{m-1}^{-1}Ca_{m-1}$ is isomorphic to $\ast^m A$. This shows that, a group $G$ is (fully) residually $A$-free, if for every non-identity element $g\in G$ (for every  finite set of non-identity elements $g_1, \ldots, g_n\in G$), there exists a homomorphism $\alpha:G\to A\ast A$ such that $\alpha(g)\neq 1$ (similarly, $\alpha(g_i)\neq 1$ for all $1\leq i\leq n$).
\end{remark}

\begin{proposition}\label{CSNk}
Let $G$ be  fully residually $A$-free group. Then $G$ is $\CSNk$.
\end{proposition}

\begin{proof}
We need to show that $G$ satisfies the first order sentences $\mathrm{Subgp}$, $\mathrm{Nil}$, and  $\mathrm{Mal}$. We start with $\mathrm{Subgp}$. Let $a, b_1, b_2\in G$ such that $a\neq 1, Q(a, b_1)$, and $Q(a, b_2)$. Let $S$ be the finite set
$$
\{ a, b_1, b_2, [x_1, \ldots, x_{k+1}]:\ \{x_1, \ldots, x_{k+1}\}=\{a, b_1^{-1}b_2\}\}.
$$
Then there exists a homomorphism $\alpha:G\to A\ast A$ such that the restriction of $\alpha$ to $S$ is injective. In the group $A\ast A$ the sentences  $Q(\alpha(a), \alpha(b_1))$ and $Q(\alpha(a), \alpha(b_2))$ are true and  hence $Q(\alpha(a), \alpha(b_1^{-1}b_2))$ is also true in  $A\ast A$ as it is a $\CSNk$. This means that $[y_1, \ldots, y_{k+1}]=1$ for all choices of
$$
y_i\in \{ \alpha(a), \alpha(b_1^{-1}b_2)\}.
$$
As the restriction of $\alpha$ to $S$ is injective, we have
$$
\bigwedge_{x_1, \ldots, x_{k+1}\in \{a, b_1^{-1}b_2\}}[x_1, \ldots, x_{k+1}]=1.
$$
This shows that $Q(a, b_1^{-1}b_2)$ is true and hence $G$ satisfies $\mathrm{Subgp}$.

Now, we show that the sentence $\mathrm{Nil}$ is true in $G$. Let $a, b_1, \ldots , b_{k+1}\in G$ such that $a$ is a  non-identity element and we have $\bigwedge_{i=1}^{k+1}Q(a, b_i)$ but $[b_1, \ldots, b_{k+1}]\neq 1$. Consider the finite set
$$
S=\{ a, b_1, \ldots, b_{k+1}, [b_1, \ldots, b_{k+1}]\}.
$$
Again, there is  a homomorphism $\alpha:G\to A\ast A$ such that the restriction of $\alpha$ to $S$ is injective. We have
$$
\bigwedge_{i=1}^{k+1}Q(\alpha(a), \alpha(b_i))
$$
and consequently $[\alpha(b_1), \ldots, \alpha(b_{k+1})]=1$ as the $A$-free group $A\ast A$ is $\CSNk$. This shows that $[b_1, \ldots, b_{k+1}]=1$, a contradiction. As a result $G$ satisfies $\mathrm{Nil}$.

A similar argument shows that the sentence $\mathrm{Mal}$ is also true in $G$ and hence $G$ is a $\CSNk$ group.
\end{proof}

B. Baumslag showed in \cite{BB} that two properties of being CT and CSA are equivalent if the given group is residually free. Our main aim in this section is to prove its general form. A theorem of Jennings (see \cite{Hall}) says that every finitely generated torsion-free nilpotent group embeds into $\mathrm{U}_r(\mathbb{Z})$, the group of uni-modular upper triangular matrices for some $r$. This implies that such a group is linear and hence it is equationally noetherian in the sense of \cite{BMR}. Now, we use the unification theorem of \cite{DMR} to prove the next lemma.

\begin{lemma}
For any positive integer $n$ the group $A^n$ is fully residually $A$-free.
\end{lemma}

\begin{proof}
We know that $\mathrm{Th}_{\forall}(A)\subseteq \mathrm{Th}_{\forall}(A^n)$ and $A^n$ is finitely generated. Now using Theorem A of \cite{DMR}, we conclude that $A^n$ is fully residually $A$, i.e. for any set of non-identity elements $u_1, \ldots, u_m\in A^n$ there exists a homomorphism $\alpha:A^n\to A$ such that $\alpha(u_i)\neq 1$ for all $1\leq i\leq m$.
\end{proof}

As a result, we have the following.

\begin{corollary}
For any set $I$ the unrestricted direct power $A^I$ is fully residually $A$-free.
\end{corollary}

\begin{proof}
Let $u_1, \ldots, u_m\in A^I$ be arbitrary elements. Then there is a finite subset $I_0\subseteq I$ such that all projections
$$
\pi_{I_0}(u_1)=(u_1)_{|_{I_0}}, \ldots, \pi_{I_0}(u_m)=(u_m)_{|_{I_0}}
$$
are non-identity. By the previous lemma, there is a homomorphism $\alpha_0:A^{I_0}\to A$ such that
$$
\alpha_0(\pi_{I_0}(u_1))\neq 1, \ldots, \alpha_0(\pi_{I_0}(u_m))\neq 1.
$$
If we consider the homomorphism $\alpha=\alpha_0\circ \pi_{I_0}: A^I\to A$, we will see that $\alpha(u_i)\neq 1$ for all $1\leq i\leq m$.
\end{proof}

\begin{lemma}
Let $G$ be a $\nilk$-group which is residually $A$-free. Then $G$ is fully residually $A$-free.
\end{lemma}

\begin{proof}
For any non-identity element $x\in  G$ there exists  a homomorphism $\alpha_x:G\to A\ast A$ such that $x$ does not belong to $\ker(\alpha_x)$. Suppose $N_x=\ker(\alpha_x)$. We know that the quotient $G/N_x$ embeds in $A\ast A$ but in the same time it is nilpotent. This shows that $G/N_x$ must be a subgroup of a conjugate of $A$ in the free product $A\ast A$. As a result $G/N_x$ embeds in $A$. Now, we have
$$
G\hookrightarrow \prod_{x\in G}\frac{G}{N_x}\hookrightarrow\prod_{x\in G}A=A^G.
$$
But we saw that $A^G$ is fully residually $A$-free, and so is $G$.
\end{proof}

In what follows $Z_k(G)$ is the $k$-th term of the upper central series of the group $G$.

\begin{lemma}
Let $G$ be a residually $A$-free group and $H$ be a normal $\nilk$-subgroup. Then $H\subseteq Z_k(G)$.
\end{lemma}

\begin{proof}
Let $H\nsubseteq Z_k(G)$ and suppose $a\in H\setminus Z_k(G)$. Then there are elements $g_1, \ldots, g_k\in G$ such that $x=[a, g_1, \ldots, g_k]\neq 1$. As $G$ is residually $A$-free, there exists a homomorphism $\alpha:G\to A\ast A$ such that $\alpha(x)\neq 1$. Hence
$$
[\alpha(a), \alpha(g_1), \ldots, \alpha(g_k)]\neq 1.
$$
This shows that $K=\mathrm{Im}(\alpha)$ is not nilpotent and so by Kurosh subgroup theorem, $K$ must be a free product of at least two factors. Now, $1\neq \alpha(H)$ is a normal subgroup of $K$ and it is $\nilk$. So, $\alpha(H)$ is a subgroup of a conjugate of some free factor of $K$. But, in this case it can not be a normal subgroup of $K$.
\end{proof}

Now we are ready to prove our main theorem.

\begin{theorem}
A group $G$ is fully residually $A$-free if and only if it is $\NTk$ and residually $A$-free.
\end{theorem}

\begin{proof}
We already proved that all fully residually $A$-free groups are $\CSNk$ (see the proposition \ref{CSNk}) and hence they are $\NTk$ and residually $A$-free. So we prove the converse. Suppose $G$ is $\NTk$ and residually $A$-free. Let $Z_k(G)\neq 1$. Then there exists an element $a\neq 1$ such that for all $g_1, \ldots, g_k\in G$ we have $[a, g_1, \ldots, g_k]=1$. Therefore, for any $b$, the subgroup $\langle a, b\rangle$ is $\nilk$. By the assumption of being $\NTk$ we see that for all $x_1, \ldots, x_{k+1}$ the subgroup $\langle a, x_1, \ldots, x_{k+1}\rangle$ is also $\nilk$ and hence
$$
[x_1, \ldots, x_{k+1}]=1.
$$
This shows that $G$ is a $\nilk$-group and hence it is fully residually $A$-free by the previous lemma.

So, we may assume that $Z_k(G)=1$ and this implies that there is no normal $\nilk$-subgroup except the identity. We proceed by induction: let for any non-identity elements $x_1, \ldots, x_{n-1}$ there exists a homomorphism $\alpha:G\to A\ast A$ such that
$$
\alpha(x_1)\neq 1, \ldots, \alpha(x_{n-1})\neq 1.
$$
Now suppose $g_1, \ldots, g_n\in G$ are non-identity. Assume that  for all $x$ we have $g_n^x\in C^k_G(g_1)$. This means that $H=\langle g_n^G\rangle$ is included in $C^k_G(g_1)$ and hence it is a normal $\nilk$-subgroup. Consequently we must have $H=1$ which is not possible as $g_n\neq 1$. Therefore, there is an element $x$ such that $\langle g_1, g_n^x\rangle$ is not $\nilk$ and hence there is a commutator
$$
[u_1, \ldots, u_{k+1}]\neq 1
$$
with $\{ u_1, \ldots, u_{k+1}\}=\{ g_1, g_n^x\}$. Consider the finite set
$$
S=\{ g_2, \ldots, g_{n-1}, [u_1, \ldots, u_{k+1}]\}.
$$
By the inductive hypothesis, there is a homomorphism $\alpha:G\to A\ast A$ such that
$$
\alpha(g_2)\neq 1, \ldots, \alpha(g_{n-1})\neq 1, [\alpha(u_1), \ldots, \alpha(u_{k+1})]\neq 1.
$$
This implies  that
$$
\alpha(g_1)\neq 1, \ldots,  \alpha(g_n)\neq 1,
$$
so $G$ is fully residually $A$-free.
\end{proof}

As a results, we have:

\begin{corollary}
Let $G$ be a $\NTk$ group. If there is a finitely generated free element $A$ in the variety $\nilk$ such that $G$ is residually $A$-free, then $G$ is $\CSNk$.
\end{corollary}

\section{Supplementary notes}

As we mentioned before, the idea of $\CSNk$ and $\NTk$ groups can be formulated in more general frame of inductive, subgroup closed classes. For example consider the class of locally nilpotent groups. One can rewrite the  proof of the proposition \ref{first} in appropriate form to conclude the following statement.

\begin{proposition}
Let every maximal locally nilpotent subgroup of a group $G$ be malnormal. Then for every pair $K_1$ and $K_2$ of locally nilpotent subgroups of $G$, $K_1\cap K_2\neq 1$ implies that $\langle K_1, K_2\rangle$ is also locally nilpotent.
\end{proposition}

It is straightforward to check that all statements of the first section are valid for the class of locally nilpotent groups after some small modifications. Let's denote this class by $\mathrm{LN}$. For an arbitrary group $G$ and every element $x\in G$, we define 
$$
C^{\ast}_G(x)=\{ y\in G:\ \langle x, y\rangle=\mathrm{nilpotent}\}.
$$
This plays the role of $C^k_G(x)$ in the first section: a group $G$ is $\mathrm{LNT}$ if and only if for every non-identity element $x$ the set $C^{\ast}_G(x)$ is a locally nilpotent subgroup of $G$. In this case, maximal locally nilpotent subgroups of $G$ are precisely the subgroups $C^{\ast}_G(x)$ for $x\neq 1$. All proofs are almost the same as in the case of the variety $\mathfrak{N}_k$. One can axiomatize the classes $\mathrm{LNT}$ and $\mathrm{CSLN}$ using universal first order sentences of infinite length. For $x_1, \ldots, x_n\in G$ consider the following infinite sentence
$$
Q^{\ast}_n(x_1,\ldots, x_n):\ \bigvee_{k=1}^{\infty}\bigwedge_{y_1, \ldots, y_{k+1}\in \{ x_1, \ldots, x_n\} }([y_1, \ldots, y_{k+1}]\approx 1).
$$
So, $Q^{\ast}_n(x_1, \ldots, x_n)$ is true in $G$ if and only if $\langle x_1, \ldots, x_n\rangle $ is  nilpotent. Now, consider the following infinite sentences.
$$
(\mathrm{Subgp}_{\ast})\ \forall x\forall y_1, y_2: ((x\not\approx 1\wedge Q_2^{\ast}(x, y_1)\wedge Q_2^{\ast}(x, y_2))\to Q_2^{\ast}(x, y_1^{-1}y_2))
$$
$$
(\mathrm{Nil}_n)\ \forall x\forall x_1, \ldots, x_n: ((x\not\approx 1)\wedge \bigwedge_{i=1}^nQ_2^{\ast}(x, x_i))\to Q_n^{\ast}(x_1, \ldots, x_n))
$$
$$
(\mathrm{Mal}_{\ast})\ \ \ \forall x, y, z: ((x, y\not \approx 1)\wedge Q_2^{\ast}(x, y)\wedge Q_2^{\ast}(x, y^z))\to Q_2^{\ast}(x, z)).
$$
Then, the class $\mathrm{LNT}$ is axiomatized by $\mathrm{Subgp}_{\ast}+\mathrm{Nil}_n$ for $n\geq 2$. The class $\mathrm{CSLN}$ is also axiomatized by $\mathrm{Subgp}_{\ast}+\mathrm{Mal}_{\ast}+\mathrm{Nil}_n$ for $n\geq 2$.

At the moment, we don't have any idea of developing results of the second section to the class $\mathrm{LN}$. It seems that the study of the same problems for other classes might be interesting, although it is not clear that if the same statements are  true for arbitrary class $\mathfrak{X}$ or not. Two special cases are the Burnside variety $\mathrm{B}_n$ and the variety of metabelian groups $\mathfrak{A}_2$. None of our arguments can be applied for these cases.

When I was finalizing this manuscript, I saw the reference \cite{Cost} where the notion of CT groups is generalized as follows: a group $G$ is called $\nilk T$ if for any three non-identity elements $x$, $y$, and $z$, the two generator subgroup $\langle x, z\rangle$ is nilpotent of class at most $k$, when $\langle x, y\rangle$ and $\langle y, z\rangle$ are so. Although this  property is weaker than our $\NTk$,  the paper \cite{Cost} provides a detailed description of finite and even locally finite $\nilk T$ groups. Consequently, \cite{Cost} can be used to study of finite and locally finite $\NTk$ groups in the forthcoming works.


\begin{thebibliography}{99}

\bibitem{BB}
Baumslag B. {\it Residually free groups}. Proc. London Math. Soc. {\bf 17}(3), 1967, pp. 402-418.

\bibitem{BMR} Baumslag G., Myasnikov A.,  Remeslennikov V.
{\it Algebraic geometry over groups: I. Algebraic sets and ideal theory}. J.  Algebra, {\bf 219}, 1999, pp. 16-79.

\bibitem{Champ}
Champetier C., Guirardel V.  {\it Limit groups as limits of free groups}. Israel Journal of Mathematics, {\bf 146}, 2005, pp. 1-75.

\bibitem{Ciob}
Ciobanu L., Fine B., Rosenberger G. {\it Classes of groups generalizing a theorem of Benjamin Baumslag}. Communications in Algebra, {\bf 44}(2), 2016, pp. 656-667.

\bibitem{Cost}
Costantino D., $\mathrm{Primo\breve{z}}$ M,. Chiara N. {\it Groups in which the bounded nilpotency of two generator subgroups is a transitive relation}. $\mathrm{Beitr\ddot{a}ge}$ zur Algebra und Geometrie, {\bf 48}(1), 2007, pp. 69-82.

\bibitem{DMR}
Daniyarova E., Myasnikov A., Remeslennikov V. {\it Algebraic geometry over algebraic structures, II: Fundations}. J. Math. Sci., 2012, {\bf 185 (3)},
pp. 389-416.

\bibitem{Fine}
Fine B., Gaglione A., Rosenberger G., Spellman D.  {\it On CT and CSA groups and related ideas}. J. Group Theory, {\bf 19}, 2016, pp. 923-940.

\bibitem{Hall}
Hall P. {\it Nilpotent groups}. Notes of Lectures given at the Canadian Math. Congress, University of Alberta, 1957.

\bibitem{Isaacs}
Isaacs M. {\it Character theory of finite groups}. AMS Chelsea Publishing, 1976.

\bibitem{MR1}
Myasnikov A., Remeslennikov V. {\it Groups with exponents I: Fundamentals of the theory and tensor completions}.
Siberian Mathematical Journal, {\bf 35}(5), 1994, pp. 986-996.

\bibitem{MR2}
Myasnikov A., Remeslennikov V. {\it Exponential groups II: Extensions of centralizers and tensor completions of CSA groups}.
International Journal of Algebra and Computations, {\bf 6}(6), 1996, pp. 678-711.


\bibitem{Rem}
Remeslennikov V.  {\it $\exists$-free groups}. Siberian Mathematical Journal, {\bf 30}(6), 1989, pp. 193-197.

\end{thebibliography}
\end{document}